\documentclass[12pt, reqno, twoside, letterpaper]{amsart}

\usepackage{beattybiasstyle}

\title[Consecutive primes and Beatty sequences]
      {Consecutive primes and Beatty sequences}
      
\author[W.\ D.\ Banks]{William D.\ Banks}

\address{Department of Mathematics, 
         University of Missouri, 
         Columbia MO, USA.}

\email{bankswd@missouri.edu}
        
\author[Victor Z.\ Guo]{Victor Z.\ Guo}

\address{Department of Mathematics, 
         University of Missouri, 
         Columbia MO, USA.}

\email{zgbmf@mail.missouri.edu}

\date{\today}

\begin{document}

\begin{abstract}
Fix irrational numbers $\alpha,\hat\alpha>1$ of finite type
and real numbers $\beta,\hat\beta\ge 0$, and let
$\cB$ and $\hat\cB$ be the Beatty sequences
$$
\cB\defeq(\fl{\alpha m+\beta})_{m\in\NN}
\mand
\hat\cB\defeq(\lfloor\hat\alpha m+\hat\beta\rfloor)_{m\in\NN}.
$$
In this note, we study the distribution of pairs $(p,p^\sharp)$
of consecutive primes for which $p\in\cB$ and $p^\sharp\in\hat\cB$.
Under a strong (but widely accepted) form of the Hardy-Littlewood
conjectures, we show that
$$
\big|\{p\le x:p\in\cB\text{~and~}p^\sharp\in\hat\cB\}\big|
=(\alpha\hat\alpha)^{-1}\pi(x)
+O\big(x(\log x)^{-3/2+\eps}\big).
$$
\end{abstract}

\maketitle

\begin{quote}
\end{quote}

\begin{quote}
\end{quote}

\newcommand{\tind}[1]{\ensuremath{\widetilde{\mathbf{1}}_{#1}}} 


\section{Introduction}
\label{sec:intro}

For any given real numbers $\alpha>0$ and  $\beta\ge 0$,
the associated (generalized) Beatty sequence is defined by
$$
\cB_{\alpha,\beta}\defeq\big(\fl{\alpha m+\beta}\big)_{m\in\NN},
$$
where $\fl{t}$ is the largest integer not exceeding $t$.
If $\alpha$ is irrational, it follows from a classical
exponential sum estimate of Vinogradov~\cite{Vino} that
$\cB_{\alpha,\beta}$ contains infinitely many prime numbers;
in fact, one has
$$
\#\big\{\text{prime~}p\le x:p\in\cB_{\alpha,\beta}\big\}\sim
\alpha^{-1}\pi(x)\qquad(x\to\infty),
$$
where $\pi(x)$ is the prime counting function.

Throughout this paper, we fix two (not necessarily distinct)
irrational numbers $\alpha,\hat\alpha>1$ and two
(not necessarily distinct) real numbers $\beta,\hat\beta\ge 0$,
and we denote
\begin{equation}
\label{eq:BBhat}
\cB\defeq\cB_{\alpha,\beta}\mand
\hat\cB\defeq\cB_{\hat\alpha,\hat\beta}.
\end{equation}
Our aim is to study the set of primes $p\in\cB$ for which
the next larger prime $p^\sharp$ lies in $\hat\cB$.  The results
we obtain are conditional, relying only on 
the \emph{Hardy-Littlewood conjectures}
in the following strong form.
Let $\cH$ be a finite subset of $\ZZ$, and let
$\ind{\PP}$ denote the indicator function of the primes.
The Hardy-Littlewood conjecture for $\cH$ asserts
that the estimate
\begin{equation}
\label{eq:HL}
\sum_{n\le x}\prod_{h\in\cH}\ind{\PP}(n+h)
=\fS(\cH)\int_2^x\frac{du}{(\log u)^{|\cH|}}
+O(x^{1/2+\eps})
\end{equation}
holds for any fixed $\eps>0$, where $\fS(\cH)$ is the
singular series given by
$$
\fS(\cH)\defeq\prod_p\bigg(1-\frac{|(\cH\bmod p)|}p\bigg)
\bigg(1-\frac1p\bigg)^{-|\cH|}.
$$
Our main result is the following.

\begin{theorem}
\label{thm:main}
Fix irrational numbers $\alpha,\hat\alpha>1$ of finite type
and real numbers $\beta,\hat\beta\ge 0$, and let
$\cB$ and $\hat\cB$ be the Beatty sequences given
by \eqref{eq:BBhat}.
For every prime $p$, let $p^\sharp$ denote the next larger prime.
Suppose that the Hardy-Littlewood conjecture~\eqref{eq:HL}
holds for every finite subset $\cH$ of $\ZZ$.
Then, for any fixed $\eps>0$, the counting function
$$
\pi(x;\cB,\hat\cB)\defeq
\big|\{p\le x:p\in\cB\text{~and~}p^\sharp\in\hat\cB\}\big|
$$
satisfies the estimate
$$
\pi(x;\cB,\hat\cB)=(\alpha\hat\alpha)^{-1}\pi(x)
+O\big(x(\log x)^{-3/2+\eps}\big),
$$
where the implied constant depends only on
$\alpha$, $\hat\alpha$ and $\eps$.
\end{theorem}

Our results are largely inspired by the
recent breakthrough paper of Lemke Oliver and
Soundararajan~\cite{LO-S}, which studies the surprisingly
erratic distribution of pairs of consecutive primes
amongst the $\phi(q)^2$ permissible reduced residue classes
modulo $q$.  In~\cite{LO-S} a conjectural explanation 
for this phenomenon is given which is based on the strong form
of the Hardy-Littlewood conjectures considered in this note, 
that is, under the hypothesis that the estimate
\eqref{eq:HL} holds for every finite subset $\cH$ of $\ZZ$.

\section{Preliminaries}

\subsection{Notation}

The notation $\llbracket t\rrbracket$ is used to denote the distance
from the real number $t$ to the nearest integer; that is,
$$
\llbracket t\rrbracket\defeq\min_{n\in\ZZ}|t-n|\qquad(t\in\RR).
$$
We denote by $\fl{t}$ and $\{t\}$ the greatest integer
$\le t$ and the fractional part of $t$, respectively.
We also write $\e(t)\defeq e^{2\pi it}$ for all $t\in\RR$, as usual.

Let $\PP$ denote the set of primes in $\NN$.
In what follows, the letter $p$ always denotes a prime number, and
$p^\sharp$ is used to denote the smallest
prime greater than $p$. In other words, $p$ and $p^\sharp$
are consecutive primes with $p^\sharp>p$. We also put
$$
\delta_p\defeq p^\sharp-p\qquad(p\in\PP).
$$

For an arbitrary set $\cS$, we use $\ind{\cS}$ to denote
its indicator function:
$$
\ind{\cS}(n)\defeq\begin{cases}
1&\quad\hbox{if $n\in\cS$,}\\
0&\quad\hbox{if $n\not\in\cS$.}\\
\end{cases}
$$

Throughout the paper, implied constants in symbols $O$, $\ll$
and $\gg$ may depend (where obvious) on the parameters
$\alpha,\hat\alpha,\eps$ but are absolute otherwise. For given
functions $F$ and $G$, the notations $F\ll G$, $G\gg F$ and
$F=O(G)$ are all equivalent to the statement that the inequality
$|F|\le c|G|$ holds with some constant $c>0$.

\subsection{Discrepancy}

We recall that the discrepancy $D(M)$ of a sequence of (not
necessarily distinct) real numbers $x_1,x_2,\ldots,x_M\in[0,1)$ is
defined by
\begin{equation}
\label{eq:descr_defn}
D(M)\defeq\sup_{\cI\subseteq[0,1)}
\left|\frac{V(\cI,M)}{M}-|\cI|\,\right|,
\end{equation}
where the supremum is taken over all intervals $\cI=(b,c)$ contained
in $[0,1)$, the quantity $V(\cI,M)$ is the number of positive
integers $m\le M$ such that $x_m\in\cI$,
and $|\cI|=c-b$ is the length of $\cI$.

For any irrational number $a$ we define its type $\tau=\tau(a)$
by the relation
$$
\tau\defeq\sup\big\{t\in\RR:\liminf\limits_{n\to\infty}
~n^t\,\llbracket a n\rrbracket=0\big\}.
$$
Using Dirichlet's approximation theorem, one sees
that $\tau\ge 1$ for every irrational number $a$. Thanks to
the work of
Khinchin~\cite{Khin} and Roth~\cite{Roth1,
Roth2} it is known that $\tau=1$ for almost all real numbers (in
the sense of the Lebesgue measure) and for all irrational
algebraic numbers, respectively.

For a given irrational number $a$,
it is well known that the sequence
of fractional parts $\{a\},\{2a\},\{3a\},\,\ldots\,,$
is uniformly distributed modulo one (see, for example,
\cite[Example~2.1, Chapter~1]{KuNi}). When $a$ is of finite
type, this statement can be made more precise. 
By~\cite[Theorem~3.2, Chapter~2]{KuNi} we have the following
result.

\begin{lemma}
\label{lem:discr_with_type}  Let $a$ be a fixed irrational
number of finite type $\tau$.  For every $b\in\RR$ the
discrepancy $D_{a,b}(M)$ of the sequence of fractional parts
$(\{am+b\})_{m=1}^M$ satisfies the bound
$$
D_{a,b}(M)\le M^{-1/\tau+o(1)}\qquad(M\to\infty),
$$
where the function implied by $o(\cdot)$ depends only on $a$.
\end{lemma}

\subsection{Indicator function of a Beatty sequence}

As in \S\ref{sec:intro} we fix (possibly equal) 
irrational numbers $\alpha,\hat\alpha>1$ and (possibly equal)
real numbers $\beta,\hat\beta\ge 0$, and we set
$$
\cB\defeq\cB_{\alpha,\beta}\mand
\hat\cB\defeq\cB_{\hat\alpha,\hat\beta}.
$$
In what follows we denote
$$
a\defeq\alpha^{-1},\qquad
\hat a\defeq\hat\alpha^{-1},\qquad
b\defeq\alpha^{-1}(1-\beta)\mand
\hat b\defeq\hat\alpha^{-1}(1-\hat\beta).
$$
It is straightforward to show that
\begin{equation}
\label{eq:ind-Beatty}
\ind{\cB}(m)=\psi_a(am+b)\mand
\ind{\hat\cB}(m)=\psi_{\hat a}(\hat am+\hat b)\qquad(m\in\NN),
\end{equation}
where for any $t\in(0,1)$ we use $\psi_t$ to denote the
periodic function of period one defined by
$$
\psi_t(x) \defeq \left\{  \begin{array}{ll}
1& \quad \hbox{if $0<\{x\}\le t$}, \\
0& \quad \mbox{if $t<\{x\}<1$ or $\{x\}=0$}.
\end{array} \right.
$$

\subsection{Modified Hardy-Littlewood conjecture}
For their work on primes in short intervals,
Montgomery and Soundararajan \cite{MontSound} have introduced
the modified singular series
$$
\fS_0(\cH)\defeq\sum_{\cT\subseteq\cH}
(-1)^{|\cH\setminus\cT|}\fS(\cT),
$$
for which one has the relation
$$
\fS(\cH)=\sum_{\cT\subseteq\cH}\fS_0(\cT).
$$
Note that $\fS(\varnothing)=\fS_0(\varnothing)=1$.
The Hardy-Littlewood conjecture \eqref{eq:HL} can be reformulated
in terms of the modified singular series as follows:
\begin{equation}
\label{eq:modHL}
\sum_{n\le x}\prod_{h\in\cH}
\bigg(\ind{\PP}(n+h)-\frac{1}{\log n}\bigg)
=\fS_0(\cH)\int_2^x\frac{du}{(\log u)^{|\cH|}}
+O(x^{1/2+\eps}).
\end{equation}

\begin{lemma}
\label{lem:G0ests}
We have
\begin{align*}
\sum_{\substack{1\le t\le h-1}}\fS_0(\{0,t\})
&\ll h^{1/2+\eps},\\
\sum_{\substack{1\le t\le h-1}}\fS_0(\{t,h\})
&\ll h^{1/2+\eps},\\
\sum_{\substack{1\le t_1<t_2\le h-1}}\fS_0(\{t_1,t_2\})
&=-\tfrac12 h\log h+\tfrac12Ah+O(h^{1/2+\eps}),
\end{align*}
where $A\defeq 2-C_0-\log 2\pi$ and $C_0$
denotes the Euler-Mascheroni constant.
\end{lemma}

\begin{proof}
Let us denote
$$
B\defeq\sum_{\substack{1\le t\le h-1}}\fS_0(\{0,t\}),\qquad
C\defeq\sum_{\substack{1\le t\le h-1}}\fS_0(\{t,h\}),
$$
and
$$
D_\pm\defeq
\sum_{\substack{1\le t_1<t_2\le h\pm 1}}\fS_0(\{t_1,t_2\})
$$
for either choice of the sign $\pm$.  Clearly,
$$
\fS_0(\{0,h\})+B+C+D_-=D_+\mand
B=\sum_{\substack{1\le t\le h-1}}\fS_0(\{0,h-t\})=C.
$$
From \cite[Equation~(16)]{MontSound} we derive the estimates
$$
D_\pm=-\tfrac12 h\log h+\tfrac12Ah+O(h^{1/2+\eps}).
$$
Using the trivial bound $\fS_0(\{0,h\})\ll\log\log h$ and
putting everything together, we finish the proof.
\end{proof}

\subsection{Technical lemmas}

Let $\nu(u)\defeq 1-1/\log u$.  Note that $\nu(u)\asymp 1$
for $u\ge 3$.

\begin{lemma}
\label{lem:Dom}
Let $c>0$ be a constant, and suppose that
$f$ is a function such that $|f(h)|\le h^c$ for all $h\ge 1$.
Then, uniformly for $3\le u\le x$ and $\lambda\in\RR$ we have
$$
\sum_{\substack{h\le(\log x)^3\\2\,\mid\,h}}
f(h)\nu(u)^h\e(\lambda h)
=\sum_{\substack{h\ge 1\\2\,\mid\,h}}
f(h)\nu(u)^h\e(\lambda h)+O_c(x^{-1}).
$$
\end{lemma}

\begin{proof}
Write $\nu(u)^h=e^{-h/H}$ with $H\defeq -(\log\nu(u))^{-1}$.
Since $H\le\log u$ for $u\ge 3$, for any $h>(\log x)^3$ we have
$h/H\ge h^{2/3}$ as $u\le x$; therefore,
\begin{align*}
\bigg|\sum_{\substack{h>(\log x)^3\\2\,\mid\,h}}
f(h)\nu(u)^h\e(\lambda h)\bigg|
\le\sum_{h>(\log x)^3}h^ce^{-h^{2/3}}
\le x^{-1}\sum_{h>(\log x)^3}h^ce^{h^{1/3}-h^{2/3}}
\ll_c x^{-1},
\end{align*}
and the result follows.
\end{proof}

The next statement is an analogue of \cite[Proposition~2.1]{LO-S}
and is proved using similar methods.

\begin{lemma}
\label{lem:RST}
Fix $\theta\in[0,1]$ and $\vartheta=0$ or $1$.
For all $\lambda\in\RR$ and $u\ge 3$, let
\begin{align*}
R_{\theta,\vartheta;\lambda}(u)&\defeq
\sum_{\substack{h\ge 1\\2\,\mid\,h}}
h^\theta(\log h)^\vartheta\nu(u)^h\e(\lambda h),\\
S_\lambda(u)&\defeq
\sum_{\substack{h\ge 1\\2\,\mid\,h}}
\fS_0(\{0,h\})\nu(u)^h\e(\lambda h).
\end{align*}
When $\lambda=0$ we have the estimates
\begin{align*}
R_{\theta,0;0}(u)&=\tfrac12\Gamma(1+\theta)
(\log u)^{1+\theta}+O(1),\\
R_{\theta,1;0}(u)&=\tfrac12(\log 2)\Gamma(1+\theta)
(\log u)^{1+\theta}+O(1),\\
S_0(u)&=\tfrac12\log u-\tfrac12\log\log u+O(1).
\end{align*}
On the other hand, if $\lambda$ is such that
$|\lambda|\ge(\log u)^{-1}$, then
$$
\max\big\{|R_{\theta,\vartheta;\lambda}(u)|,
|S_\lambda(u)|\big\}\ll\lambda^{-4}.
$$
\end{lemma}

\begin{proof}
We adapt the proof of \cite[Proposition~2.1]{LO-S}.
As in Lemma~\ref{lem:Dom} we write
$\nu(u)^h=e^{-h/H}$ with $H\defeq -(\log\nu(u))^{-1}$.
We simplify the expressions $R_{\theta,\vartheta;\lambda}(u)$,
$S_\lambda(u)$ and $T_\lambda(u)$ by writing
$$
\nu(u)^h\e(\lambda h)=e^{-h/H_\lambda}
\qquad\text{with}\quad H_\lambda\defeq\frac{H}{1-2\pi i\lambda H}.
$$
Since $\Re(h/H_\lambda)=h/H>0$ for any positive integer $h$,
using the Cahen-Mellin integral we have
$$
R_{\theta,\vartheta;\lambda}(u)
=\sum_{\substack{h\ge 1\\2\,\mid\,h}}
h^\theta(\log h)^\vartheta e^{-h/H_\lambda}
=\frac{1}{2\pi i}\int_{4-i\infty}^{4+i\infty}
\bigg(\sum_{\substack{h\ge 1\\2\,\mid\,h}}
\frac{h^\theta(\log h)^\vartheta}{h^s}\bigg)
\Gamma(s)H_\lambda^s\,ds.
$$
In particular,
\begin{equation}
\label{eq:CahMel_0A}
R_{\theta,0;\lambda}(u)
=\frac{2^\theta}{2\pi i}
\int_{4-i\infty}^{4+i\infty}2^{-s}\zeta(s-\theta)
\Gamma(s)H_\lambda^s\,ds
\end{equation}
and
\begin{equation}
\label{eq:CahMel_1A}
R_{\theta,1;\lambda}(u)
=R_{\theta,0;\lambda}(u)\log2-\frac{2^\theta}{2\pi i}
\int_{4-i\infty}^{4+i\infty}2^{-s}
\zeta'(s-\theta)\Gamma(s)H_\lambda^s\,ds.
\end{equation}
When $\lambda\ne 0$ we have
\begin{align*}
\big|R_{\theta,0;\lambda}(u)\big|
&\le\frac{2^{\theta-4}|H_\lambda|^4}{2\pi}
\int_{-\infty}^\infty\big|\zeta(4-\theta+it)\Gamma(4+it)\big|\,dt\\
&\ll |H_\lambda|^4=\bigg(\frac{H^2}{1+4\pi^2\lambda^2H^2}\bigg)^2,
\end{align*}
hence the bound $R_{\theta,0;\lambda}(u)\ll \lambda^{-4}$ holds 
if $|\lambda|\ge(\log u)^{-1}$ since
$H\asymp\log u$ for $u\ge 3$.
In the case that $\lambda=0$, the stated estimate for
$R_{\theta,0;0}(u)$ is obtained by
shifting the line of integration in \eqref{eq:CahMel_0A}
to the line $\{\Re(s)=-\tfrac13\}$ (say), taking into
account the residues of the poles of the integrand
at $s=1+\theta$ and $s=0$.

Our estimates for $R_{\theta,1;\lambda}(u)$ are proved similarly,
using \eqref{eq:CahMel_1A} instead of \eqref{eq:CahMel_0A}
and taking into account that
$\zeta'(s-\theta)=(s-1-\theta)^{-1}+O(1)$ for $s$ near $1+\theta$.

Next, for all $\lambda\in\RR$ and $u\ge 3$, let
$$
T_\lambda(u)\defeq
\sum_{h\ge 1}\fS(\{0,h\})\,e^{-h/H_\lambda}.
$$
Since $\fS_0(\{0,h\})=\fS(\{0,h\})-1$ for all integers $h$,
and $\fS(\{0,h\})=0$ if $h$ is odd, it follows that
$$
S_\lambda(u)=T_\lambda(u)-R_{0,0;\lambda}(u)
=T_\lambda(u)-\tfrac12\log u+O(1).
$$
Hence, to complete the proof of the lemma, it suffices
to show that
$$
T_0(u)=\log u-\tfrac12\log\log u+O(1)\mand
T_\lambda(u)\ll\lambda^{-4}\text{~if~}|\lambda|\ge(\log u)^{-1}.
$$

As in the proof of \cite[Proposition~2.1]{LO-S}, we consider the
Dirichlet series
$$
F(s)\defeq\sum_{h\ge 1}\frac{\fS(\{0,h\})}{h^s},
$$
which can be expressed in the form
$$
F(s)=\frac{\zeta(s)\zeta(s+1)}{\zeta(2s+2)}
\prod_p\(1-\frac{1}{(p-1)^2}+\frac{2p}{(p-1)^2(p^{s+1}+1)}\),
$$
and the final product is analytic for $\Re(s)>-1$.
Using the Cahen-Mellin integral we have
\begin{equation}
\label{eq:CahMel2}
T_\lambda(u)
=\frac{1}{2\pi i}\int_{4-i\infty}^{4+i\infty}
F(s)\Gamma(s)H_\lambda^s\,ds.
\end{equation}
For $\lambda\ne 0$ we have
$$
\big|T_\lambda(u)\big|\le\frac{|H_\lambda|^4}{2\pi}
\int_{-\infty}^\infty\big|F(4+it)\Gamma(4+it)\big|\,dt
\ll |H_\lambda|^4=\bigg(\frac{H^2}{1+4\pi^2\lambda^2H^2}\bigg)^2
$$
hence $T_\lambda(u)\ll \lambda^{-4}$ holds provided that
$|\lambda|\ge(\log u)^{-1}$. For $\lambda=0$,
we shift the line of integration in \eqref{eq:CahMel2}
to the line $\{\Re(s)=-\tfrac13\}$ (say), taking into account
the double pole at $s=0$ and the simple pole at $s=1$.
This leads to the stated estimate for $T_0(u)$.
\end{proof}

We also need the following integral estimate (proof omitted).

\begin{lemma}
\label{lem:card}
For all $\lambda\in\RR$ and $x\ge 3$, let
$$
I_\lambda(x)\defeq\int_3^x\frac{\e(\lambda u)}{\nu(u)\log u}\,du.
$$
When $\lambda=0$ we have the estimate
$$
I_0(x)=\frac{x}{\log x}+O\bigg(\frac{x}{(\log x)^2}\bigg),
$$
whereas for any $\lambda\ne 0$ we have
$$
I_\lambda(x)\ll|\lambda|^{-1}.
$$
\end{lemma}

\section{Proof of Theorem~\ref{thm:main}}

For every even integer $h\ge 2$ we denote
$$
\pi_h(x;\cB,\hat\cB)\defeq\big|\{p\le x:p\in\cB,~p^\sharp\in\hat\cB
\text{~and~}\delta_p=h\}\big|
=\sum_{n\le x}\ind{\cB}(n)\ind{\hat\cB}(n+h)f_h(n),
$$
where
$$
f_h(n)\defeq\ind{\PP}(n)\ind{\PP}(n+h)
\prod_{0<t<h}\big(1-\ind{\PP}(n+t)\big)
=\begin{cases}
1&\quad\hbox{if $n=p\in\PP$ and $\delta_p=h$},\\
0&\quad\hbox{otherwise}.
\end{cases}
$$
Clearly,
\begin{equation}
\label{eq:piBBest}
\pi(x;\cB,\hat\cB)
=\sum_{\substack{h\le(\log x)^3\\2\,\mid\,h}}\pi_h(x;\cB,\hat\cB)
+O\(\frac{x}{(\log x)^3}\).
\end{equation}

Fixing an even integer $h\in[1,(\log x)^3]$ for the moment,
our initial goal is to express $\pi_h(x;\cB,\hat\cB)$
in terms of the function
$$
S_h(x)\defeq\sum_{n\le x}f_h(n)
$$
recently introduced by Lemke Oliver and
Soundararajan~\cite[Equation~(2.5)]{LO-S}.
In view of \eqref{eq:ind-Beatty} we can write
\begin{equation}
\label{eq:paris}
\pi_h(x;\cB,\hat\cB)=\sum_{n\le x}
\psi_a(an+b)\psi_{\hat a}(\hat a(n+h)+\hat b)f_h(n).
\end{equation}
According to a classical result of Vinogradov (see~\cite[Chapter~I,
Lemma~12]{Vin}), for any $\Delta$ such that
$$
0 < \Delta < \tfrac18 \mand \Delta\le
\tfrac12\min\{a,1-a\}
$$
there is a real-valued function $\Psi_a$ with the following
properties:
\begin{itemize}
\item[$(i)$~~] $\Psi_a$ is periodic with period one;

\item[$(ii)$~~] $0 \le\Psi_a(t)\le 1$ for all $t\in\RR$;

\item[$(iii)$~~] $\Psi_a(t)=\psi_a(t)$ if $\Delta\le \{t\}\le
a-\Delta$ or if $a+\Delta\le \{t\}\le 1-\Delta$;

\item[$(iv)$~~] $\Psi_a$ is represented by a Fourier series
$$
\Psi_a(t)=\sum_{k\in\ZZ}g_a(k)\e(kt),
$$
where $g_a(0)=a$, and the Fourier coefficients
satisfy the uniform bound
\begin{equation}
\label{eq:coeffbounds}
|g_a(k)|\ll\min\big\{|k|^{-1},|k|^{-2}\Delta^{-1}\big\}
\qquad(k\ne 0).
\end{equation}
\end{itemize}
For convenience, we denote
$$
\cI_a\defeq[0,\Delta)\cup(a-\Delta,a+\Delta) \cup(1-\Delta,1),
$$
so that $\Psi_a(t)=\psi_a(t)$ whenever $\{t\}\not\in\cI_a$.
Defining $\Psi_{\hat a}$ and $\cI_{\hat a}$ similarly with $\hat a$
in place of $a$, and taking into account
the properties $(i)$--$(iii)$,
from \eqref{eq:paris} we deduce that
\begin{equation}
\label{eq:basic_estimate}
\pi_h(x;\cB,\hat\cB)=\sum_{n\le x}\Psi_{a}(a n+b)
\Psi_{\hat a}(\hat a(n+h)+\hat b)f_h(n)+O(V(x)),
\end{equation}
where $V(x)$ is the number of positive integers $n\le x$ for which
$$
\{a n+b\}\in\cI_a\qquad\text{or}\qquad
\{\hat a(n+h)+\hat b\}\in\cI_{\hat a}.
$$
Since $\cI_a$ and $\cI_{\hat a}$ are unions of intervals
with overall measure $4\Delta$, it follows from the
definition \eqref{eq:descr_defn} and
Lemma~\ref{lem:discr_with_type} that
\begin{equation}
\label{eq:bound V(I,x)}
V(x)\ll\Delta x+x^{1-1/\tau+o(1)}\qquad(x\to\infty).
\end{equation}

Now let $K\ge\Delta^{-1}$ be a large real number,
and let $\Psi_{a,K}$ be the trigonometric polynomial given by
$$
\Psi_{a,K}(t)\defeq\sum_{|k|\le K}g_a(k)\e(kt).
$$
Using \eqref{eq:coeffbounds} it is clear that the estimate
\begin{equation}
\label{eq:PKP} \Psi_a(t)=\Psi_{a,K}(t)+O(K^{-1}\Delta^{-1})
\end{equation}
holds uniformly for all $t\in\RR$.
Defining $\Psi_{\hat a,K}$ in a similar way, combining
\eqref{eq:PKP} with \eqref{eq:basic_estimate}, and taking into
account \eqref{eq:bound V(I,x)}, we derive the estimate
$$
\pi_h(x;\cB,\hat\cB)=\Sigma_h
+O\big(\Delta x+x^{1-1/\tau+\eps}+K^{-1}\Delta^{-1}x\big),
$$
where
\begin{align*}
\Sigma_h&\defeq\sum_{n\le x}\Psi_{a,K}(a n+b)
\Psi_{\hat a,K}(\hat a(n+h)+\hat b)f_h(n)\\
&=\sum_{n\le x}\sum_{|k|,|\ell |\le K}
g_a(k)\e(k(a n+b))
g_{\hat a}(\ell)\e(\ell (\hat a(n+h)+\hat b))f_h(n)\\
&=\sum_{|k|,|\ell |\le K}g_a(k)\e(kb)
g_{\hat a}(\ell)\e(\ell\hat b)\cdot\e(\ell\hat ah)
\sum_{n\le x}\e((ka+\ell\hat a)n)f_h(n).
\end{align*}
Therefore
\begin{align}
\pi_h(x;\cB,\hat\cB)&=\sum_{|k|,|\ell |\le K}g_a(k)\e(kb)
g_{\hat a}(\ell)\e(\ell\hat b)\cdot\e(\ell\hat ah)
\int_{3^-}^x\e((ka+\ell\hat a)u)\,d(S_h(u))\nonumber\\
\label{eq:milk}
&\qquad+O\big(\Delta x+x^{1-1/\tau+\eps}+K^{-1}\Delta^{-1}x\big),
\end{align}
which completes our initial goal of
expressing $\pi_h(x;\cB,\hat\cB)$
in terms of the function $S_h$.
To proceed further, it is useful to recall
certain aspects of the analysis
of $S_h$ that is carried out in \cite{LO-S}.
First, writing $\tind{\PP}(n)\defeq\ind{\PP}(n)-1/\log n$, up to
an error term of size 
$O(x^{1/2+\eps})$ the quantity $S_h(x)$ is equal to
\begin{align*}
&\sum_{n\le x}\(\tind{\PP}(n)+\frac{1}{\log n}\)
\(\tind{\PP}(n+h)+\frac{1}{\log n}\)
\prod_{0<t<h}\(1-\frac{1}{\log n}-\tind{\PP}(n+t)\)\\
&\quad=\sum_{\cA\subseteq\{0,h\}}\sum_{\cT\subseteq[1,h-1]}
(-1)^{|\cT|}\sum_{n\le x}\(\frac{1}{\log n}\)^{2-|\cA|}
\(1-\frac{1}{\log n}\)^{h-1-|\cT|}
\prod_{t\in\cA\cup\cT}\tind{\PP}(n+t);
\end{align*}
see~\cite[Equations~(2.5) and (2.6)]{LO-S}.
By the modified Hardy-Littlewood conjecture \eqref{eq:modHL}
the estimate
\begin{align*}
\sum_{n\le x}(\log n)^{-c}
\prod_{t\in\cH}\tind{\PP}(n+t)
&=\int_{3^-}^x(\log u)^{-c}\,d\(\,\sum_{n\le u}
\prod_{t\in\cH}\tind{\PP}(n+t)\)\\
&=\fS_0(\cH)\int_3^x(\log u)^{-c-|\cH|}\,du
+O(x^{1/2+\eps})
\end{align*}
holds uniformly for any constant $c>0$;
consequently, up to an error term of size
$O(x^{1/2+\eps})$ the quantity $S_h(x)$ is equal to
$$
\sum_{\cA\subseteq\{0,h\}}\sum_{\cT\subseteq[1,h-1]}
(-1)^{|\cT|} \fS_0(\cA\cup\cT)
\int_3^x(\log u)^{-2-|\cT|}
\nu(u)^{h-1-|\cT|}\,du,
$$
where
$$
\nu(u)\defeq 1-\frac1{\log u}\qquad(u>1)
$$
(note that $\nu(u)$ is the same as $\alpha(u)$
in the notation of \cite{LO-S}).
For every integer $L\ge 0$ we denote
$$
\cD_{h,L}(u)\defeq
\mathop{\sum_{\cA\subseteq\{0,h\}}\sum_{\cT\subseteq[1,h-1]}}
\limits_{(|\cA|+|\cT|=L)}
(-1)^{|\cT|} \fS_0(\cA\cup\cT)
(\nu(u)\log u)^{-|\cT|}
\nu(u)^h,
$$
so that
$$
S_h(x)=\sum_{L=0}^{h+1}\int_3^x\nu(u)^{-1}(\log u)^{-2}
\cD_{h,L}(u)\,du+O(x^{1/2+\eps}).
$$
We now combine this relation with \eqref{eq:milk},
sum over the even natural numbers $h\le(\log x)^3$,
and apply \eqref{eq:piBBest} to deduce that the
quantity $\pi(x;\cB,\hat\cB)$ is equal to
\begin{align*}
&\sum_{\substack{h\le(\log x)^3\\2\,\mid\,h}}\sum_{L=0}^{h+1}
\sum_{|k|,|\ell |\le K}g_a(k)\e(kb)
g_{\hat a}(\ell)\e(\ell\hat b)\cdot\e(\ell\hat ah)
\int_3^x\frac{\e((ka+\ell\hat a)u)}{\nu(u)(\log u)^2}
\cD_{h,L}(u)\,du
\end{align*}
up to an error term of size
$$
\ll\frac{x}{(\log x)^3}+
\big(\Delta x+x^{1-1/\tau+\eps}
+K^{-1}\Delta^{-1}x\big)(\log x)^3.
$$
Choosing $\Delta\defeq(\log x)^{-6}$ and $K\defeq(\log x)^{12}$
the combined error is $O(x/(\log x)^3)$, which is acceptable.

Next, arguing as in
\cite{LO-S} and noting that
$$
\sum_{|k|,|\ell |\le K}|g_a(k)g_{\hat a}(\ell)|\ll(\log\log x)^2,
$$
one sees that the contribution to $\pi(x;\cB,\hat\cB)$ coming from
terms with $L\ge 3$ does not exceed $O(x/(\log x)^{5/2})$.
Since $\cD_{h,1}$ is identically zero (as $\fS_0$
vanishes on singleton sets), 
this leaves only the terms with $L=0$ or $L=2$.
The function $\cD_{h,2}$ splits naturally into four pieces
according to whether $\cA=\varnothing$, $\{0\}$, $\{h\}$ or
$\{0,h\}$. Consequently, up to $O(x/(\log x)^{5/2})$ we can
express the quantity $\pi(x;\cB,\hat\cB)$ as
\begin{equation}
\label{eq:vegetables}
\sum_{j=1}^5
\sum_{|k|,|\ell |\le K}g_a(k)\e(kb)
g_{\hat a}(\ell)\e(\ell\hat b)
\int_3^x\frac{\e((ka+\ell\hat a)u)}{\nu(u)(\log u)^2}
\cF_{j,\ell}(u)\,du,
\end{equation}
where (taking into account Lemma~\ref{lem:Dom}) we have written
$$
\sum_{\substack{h\le(\log x)^3\\2\,\mid\,h}}
\e(\ell\hat ah)\cD_{h,L}(u)
=\sum_{j=1}^5\cF_{j,\ell}(u)+O(x^{-1})
$$
with
\begin{align*}
\cF_{1,\ell}(u)
&\defeq\sum_{\substack{h\ge 1\\2\,\mid\,h}}
\nu(u)^h\e(\ell\hat ah),\\
\cF_{2,\ell}(u)
&\defeq\sum_{\substack{h\ge 1\\2\,\mid\,h}}
\fS_0(\{0,h\})\nu(u)^h\e(\ell\hat ah),\\
\cF_{3,\ell}(u)
&\defeq\frac{(-1)}{\nu(u)\log u}
\sum_{\substack{h\ge 1\\2\,\mid\,h}}
\sum_{\substack{1\le t\le h-1}}
\fS_0(\{0,t\})\nu(u)^h\e(\ell\hat ah),\\
\cF_{4,\ell}(u)
&\defeq\frac{(-1)}{\nu(u)\log u}
\sum_{\substack{h\ge 1\\2\,\mid\,h}}
\sum_{\substack{1\le t\le h-1}}
\fS_0(\{t,h\})\nu(u)^h\e(\ell\hat ah),\\
\cF_{5,\ell}(u)
&\defeq\frac{1}{(\nu(u)\log u)^2}
\sum_{\substack{h\ge 1\\2\,\mid\,h}}
\sum_{\substack{1\le t_1<t_2\le h-1}}
\fS_0(\{t_1,t_2\})\nu(u)^h\e(\ell\hat ah).
\end{align*}

First, we show that certain terms in \eqref{eq:vegetables}
make a negligible contribution that does not exceed
$O(x/(\log x)^{3/2-\eps})$.

For any $\ell\ne 0$, using Lemma~\ref{lem:RST} with
$\lambda=\ell\hat a$ we have
$$
\cF_{1,\ell}(u)=R_{0,0;\ell\hat a}(u)\ll\ell^{-4}
$$
provided that $|\ell\hat a|\ge(\log u)^{-1}$, and for this
it suffices that $u\ge \exp(\hat\alpha)$.  Thus,
$$
\int_3^x\frac{\e((ka+\ell\hat a)u)}{\nu(u)(\log u)^2}
\cF_{1,\ell}(u)\,du\ll
1+\ell^{-4}\frac{x}{(\log x)^2}.
$$
In view of \eqref{eq:coeffbounds}, the contribution to
\eqref{eq:vegetables} from terms with $j=1$ and $\ell\ne 0$ is
$$
\ll\sum_{\substack{|k|,|\ell |\le K\\\ell\ne 0}}
|g_a(k)|\cdot|\ell|^{-1}
\bigg(1+\ell^{-4}\frac{x}{(\log x)^2}\bigg)
\ll\frac{x\log\log x}{(\log x)^2}
\ll\frac{x}{(\log x)^{3/2-\eps}}.
$$

Similarly, for $\ell\ne 0$ and $u\ge \exp(\hat\alpha)$
we have $\cF_{2,\ell}(u)=S_{\ell\hat a}(u)\ll\ell^{-4}$
by Lemma~\ref{lem:RST}, so the contribution to
\eqref{eq:vegetables} from terms with $j=2$ and $\ell\ne 0$ is
also $O(x/(\log x)^{3/2-\eps})$.

For any $\ell\in\ZZ$, by Lemma~\ref{lem:G0ests} and
Lemma~\ref{lem:RST} we have
$$
\max\big\{\big|\cF_{3,\ell}(u)\big|,\big|\cF_{4,\ell}(u)\big|\big\}
\ll\frac{1}{\log u}\sum_{\substack{h\ge 1\\2\,\mid\,h}}
h^{1/2+\eps/2}\nu(u)^h
\ll (\log u)^{1/2+\eps/2},
$$
hence for $j=3,4$ we see that
$$
\int_3^x\frac{\e((ka+\ell\hat a)u)}{\nu(u)(\log u)^2}
\cF_{j,\ell}(u)\,du\ll\frac{x}{(\log x)^{3/2-\eps/2}}.
$$
By \eqref{eq:coeffbounds}, it follows that
the contribution to
\eqref{eq:vegetables} from terms with $j=3,4$ is
$$
\ll\frac{x}{(\log x)^{3/2-\eps/2}}
\sum_{|k|,|\ell |\le K}|g_a(k)g_{\hat a}(\ell)|
\ll\frac{x(\log\log x)^2}{(\log x)^{3/2-\eps/2}}
\ll\frac{x}{(\log x)^{3/2-\eps}}.
$$

Finally, for any $\ell\in\ZZ$ and  $u\ge \exp(\hat\alpha)$,
by Lemma~\ref{lem:G0ests} and Lemma~\ref{lem:RST} we have
\begin{align*}
\cF_{5,\ell}(u)
&=\frac{1}{(\nu(u)\log u)^2}
\sum_{\substack{h\ge 1\\2\,\mid\,h}}
\(-\tfrac12 h\log h+\tfrac12Ah+O(h^{1/2+\eps/2})\)
\nu(u)^h\e(\ell\hat ah)\\
&=\frac{-\tfrac12R_{1,1;\ell\hat a}(u)
+\tfrac12AR_{1,0;\ell\hat a}(u)
+O(R_{1/2+\eps/2,0;0}(u))}{(\nu(u)\log u)^2}\\
&\ll\frac{\lambda^{-4}+(\log u)^{3/2+\eps/2}}{(\log u)^2},
\end{align*}
and arguing as before we see that the
contribution to \eqref{eq:vegetables} coming
from terms with $j=5$ does not exceed $O(x/(\log x)^{3/2-\eps})$.

Applying the preceding bounds to \eqref{eq:vegetables}
we see that, up to $O(x/(\log x)^{3/2-\eps})$,
the quantity $\pi(x;\cB,\hat\cB)$ is equal to
$$
\hat a\sum_{j=1,2}
\sum_{|k|\le K}g_a(k)\e(kb)
\int_3^x\frac{\e(kau)}{\nu(u)(\log u)^2}
\cF_{j,0}(u)\,du,
$$
where we have used the fact that $g_{\hat a}(0)=\hat a$.
By Lemma~\ref{lem:RST} we have
$$
\cF_{1,0}(u)
=\sum_{\substack{h\ge 1\\2\,\mid\,h}}\nu(u)^h
=R_{0,0;0}(u)=\tfrac12\log u+O(1)
$$
and
$$
\cF_{2,0}(u)
=\sum_{\substack{h\ge 1\\2\,\mid\,h}}
\fS_0(\{0,h\})\nu(u)^h
=S_0(u)=\tfrac12\log u-\tfrac12\log\log u+O(1);
$$
therefore,
$$
\int_3^x\frac{\e(kau)}{\nu(u)(\log u)^2}
\cF_{j,0}(u)\,du
=\frac12\int_3^x\frac{\e(kau)}{\nu(u)\log u}\,du
+O\bigg(\frac{x\log\log x}{(\log x)^2}\bigg)\qquad(j=1,2).
$$
Consequently, up to $O(x/(\log x)^{3/2-\eps})$
we can express the quantity $\pi(x;\cB,\hat\cB)$ as
\begin{equation}
\label{eq:grains}
\hat a\sum_{|k|\le K}g_a(k)\e(kb)
\int_3^x\frac{\e(kau)}{\nu(u)\log u}\,du.
\end{equation}
To complete the proof of Theorem~\ref{thm:main},
we apply Lemma~\ref{lem:card}, which shows that the term $k=0$
in \eqref{eq:grains} contributes
$$
a\hat a\,\frac{x}{\log x}+O\bigg(\frac{x}{(\log x)^2}\bigg)
=(\alpha\hat\alpha)^{-1}\pi(x)+O\bigg(\frac{x}{(\log x)^2}\bigg)
$$
to the quantity $\pi(x;\cB,\hat\cB)$ (and thus accounts
for the main term), whereas the terms in
\eqref{eq:grains} with $k\ne 0$ contribute altogether
only a bounded amount.

\bigskip

\noindent{\bf Acknowledgement.} The first author was supported in part by
a grant from the University of Missouri Research Board.

\end{document}